\documentclass[preprint,11pt]{elsarticle}
\usepackage{amsthm,amsfonts,amssymb,amscd,amsmath,enumerate,verbatim,calc,graphicx,geometry}
\usepackage[all]{xy}
\newtheorem{theorem}{THEOREM}[section]

\newtheorem{proposition}[theorem]{PROPOSITION}
\newtheorem{corollary}[theorem]{COROLLARY}
\theoremstyle{definition}
\theoremstyle{definitions}

\newtheorem{example}[theorem]{EXAMPLE}
\theoremstyle{notations}

\theoremstyle{remarks}

\newcommand{\al}{\alpha}

\newcommand{\bt}{\beta}

\journal{}
\begin{document}

\begin{frontmatter}



\title{SOME RESULTS IN QUASITOPOLOGICAL HOMOTOPY GROUPS}

\author[label2]{T.~NASRI}
\ead{tnasri72@yahoo.com}
\author[label1]{H.~MIREBRAHIMI}
\ead{h\_mirebrahimi@um.ac.ir}
\author[label1]{H.~TORABI\corref{cor1}}
\ead{h.torabi@um.ac.ir}

\address[label1]{Department of Pure Mathematics, Center of Excellence in Analysis on Algebraic Structures, Ferdowsi University of Mashhad,\\
P.O.Box 1159-91775, Mashhad, Iran.}
\address[label2]{Department of Pure Mathematics, Faculty of Basic Sciences,  University of Bojnord,\\
 Bojnord, Iran.}
\cortext[cor1]{Corresponding author}

\begin{abstract}
{
\small
In this paper we show that the $n$th quasitopological homotopy group of a topological space is isomorphic to $(n-1)$th quasitopological homotopy group of its loop space and by this fact we obtain some results about quasitopological homotopy groups. Finally, using the long exact sequence of a based pair and a fibration in qTop introduced by Brazas in 2013, we obtain some results in this field.}
\end{abstract}


\end{frontmatter}

\section{\bf INTRODUCTION }
Endowed with the quotient topology induced by the natural surjective map $q:\Omega^n(X,x)\rightarrow \pi_n(X,x)$, where $\Omega^n(X,x)$ is the $n$th loop space of $(X,x)$ with the compact-open topology, the familiar homotopy group $\pi_n(X,x)$ becomes a quasitopological group which is called the quasitopological $n$th homotopy group of the pointed space $(X,x)$, denoted by $\pi_n^{qtop}(X,x)$ (see \cite{B,Br,Bra,G1}).

It was claimed by Biss \cite{B} that $\pi_1^{qtop}(X,x)$ is a topological group. However, Calcut and McCarthy \cite{CM} and Fabel \cite{F2} showed that there is a gap in the proof of \cite[Proposition 3.1]{B}. The misstep in the proof is repeated by Ghane et al. \cite{G1} to prove that $\pi_n^{qtop}(X,x)$ is a topological group \cite[Theorem 2.1]{G1} (see also \cite{CM}).

Calcut and McCarthy \cite{CM} showed that $\pi_1^{qtop}(X,x)$ is a homogeneous space and more precisely, Brazas \cite{Br} mentioned that $\pi_1^{qtop}(X,x)$ is a quasitopological group in the sense of \cite{A}.

Calcut and McCarthy \cite{CM} proved that for a path connected and locally path connected space $X$,  $\pi_1^{qtop}(X)$ is a discrete topological group if and only if $X$ is semilocally 1-connected (see also \cite{Br}). Pakdaman et al. \cite{P1} showed that for a locally $(n-1)$-connected space $X$, $\pi_n^{qtop}(X,x)$ is discrete if and only if $X$ is semilocally n-connected at $x$ (see also \cite{G1}). Also, they proved that the quasitopological fundamental group of every small loop space is an indiscrete topological group. We recall that a loop in X at x is called small if it is homotopic to a loop in every neighborhood
U of x. Also the topological space $X$ with non trivial fundamental group is called a small loop space if every loop of $X$ is small.

In this paper, we obtain some results about quasitopological homotopy groups. One of the main results of Section 2 is as follows:

{\textbf{Theorem 2.1}}. Let $(X,x)$ be a pointed  topological space.  Then for all $n\geq 1$ and $1\leq k\leq n-1$,
\[\pi_n^{qtop}(X,x)\cong \pi_{n-k}^{qtop}(\Omega^k(X,x), e_x).\]
where $e_x$ is the constant $k$-loop in $X$ at $x$.\\

 By this fact we can show that some properties of a space can
be transferred to its loop space. Also, we obtain several results in quasitopological homotopy groups. Moreover, we show that for a fibration $p:E\longrightarrow X$ with fiber $F$,   the induced map $f_*: \pi_{n}^{qtop}(B,b_0)\longrightarrow  \pi_{n-1}^{qtop}(F,\tilde{b_0})$ is continuous.

Brazas in his thesis \cite{Br3} exhibited two long exact sequences of based pair $(X,A)$ and fibration $p:E\longrightarrow X$ in qTop. In Section 3, we use these sequences and obtain some results in this filed. For instance, we conclude the following results:

{\textbf{Proposition 3.3}}. If $r:X\longrightarrow A$ is a retraction, then there are isomorphisms in quasitopological groups, for all $n\geq 2$,
\[\pi_n^{qtop}(X)\cong \pi_n^{qtop}(A)\times \pi_n^{qtop}(X,A).\]

{\textbf{Corollary 3.7}}. If $p:\tilde{X}\longrightarrow X$ is a covering projection, then for all $n\geq 2$, $\pi_n^{qtop}(\tilde{X})\cong \pi_n^{qtop}(X)$ and $\pi_1^{qtop}(\tilde{X})$ can be embedded in $\pi_1^{qtop}(X)$. 

\section{\bf QUASITOPOLOGICAL HOMOTOPY GROUPS}
It is well-known that for a pointed topological space  $(X,x)$, for all $n\geq 1$ and $1\leq k\leq n-1$, $\pi_n(X,x)\cong \pi_{n-k}(\Omega^k (X, x), e_x)$. In this section we extend this result for quasitopological homotopy groups and we obtain some results about them. The following theorem is one of the main results of this paper.
\begin{theorem}\label{q1}
Let $(X,x)$ be a pointed  topological space.  Then for all $n\geq 1$ and $1\leq k\leq n-1$,
\[\pi_n^{qtop}(X,x)\cong \pi_{n-k}^{qtop}(\Omega^k (X,x), e_x).\]
where $e_x$ is the constant $k$-loop in $X$ at $x$.
\end{theorem}
\begin{proof}
Consider the following commutative diagram:
\begin{equation}\label{dia}\begin{CD}
\Omega^{n} (X,x)@>\phi>>\Omega^{n-k}(\Omega^k (X,x), e_x)\\
@VV q V@V qVV\\
\pi_n^{qtop}(X,x)@>\phi_*>>\pi_{n-k}^{qtop}(\Omega^k (X,x), e_x),
\end{CD}\end{equation}
where $\phi:\Omega^{n} (X,x)\longrightarrow\Omega^{n-k}(\Omega^k (X,x), e_x)$ given by $\phi(f)=f^{\sharp}$ is a homeomorphism with inverse $g\longmapsto g^{\flat}$ in the sense of \cite{Ro}. Since the map $q$ is a quotient map, the homomorphism $\phi_*$ is an isomorphism between quasitopological homotopy groups.
\end{proof}

The following result is a consequence of Theorem \ref{q1}.

\begin{corollary}
Let $X$ be a locally $(n-1)$-connected. Then $X$ is semilocally $n$-connected at $x$ if and only if $\Omega^{n-1}(X,x)$ is semilocally simply connected at $e_x$, where   $e_x$ is the constant loop in $X$ at $x$.
\end{corollary}
\begin{proof}
Since $X$ is a locally $(n-1)$-connected, by \cite[Theorem 6.7]{P1}, $X$ is semilocally $n$-connected at $x$ if and only if  $\pi_n^{qtop}(X,x)$ is discrete. By Theorem \ref{q1}, $\pi_n^{qtop}(X,x)\cong \pi_1^{qtop}(\Omega^{n-1}(X,x), e_x)$. Also  $\pi_1^{qtop}(\Omega^{n-1}(X,x), e_x)$ is discrete if and only if $\Omega^{n-1}(X,x)$ is semilocally simply connected  at $e_x$ by \cite[Theorem 6.7]{P1}.
\end{proof}

Note that the above result has been shown by  Hidekazu Wada \cite [Remark]{Wa} and Authors \cite[Lemma 3.1]{Na} with another methods.

\begin{corollary}
Let $(X,x)=\displaystyle{\lim_{\leftarrow}(X_i,x_i)}$  be the inverse limit  of an inverse system $\{(X_i,x_i),\varphi_{ij}\}_I$. Then for all $n\geq 1$ and $1\leq k\leq n-1$,
\[\pi_n^{qtop}(X,x)\cong \pi_{n-k}^{qtop}(\displaystyle{\lim_{\leftarrow}\Omega^k(X_i,x_i)},e_x).\]
\end{corollary}
 Virk \cite{V} introduced the SG (small generated) subgroup of fundamental group $\pi_1(X,x)$, denoted by $\pi_1^{sg}(X,x)$, as the subgroup generated by the following elements $$[\al*\bt*\al^{-1}],$$
where $\al$ is a path in $X$ with initial point $x$ and $\bt$ is a small loop in $X$ at $\al(1)$ .
Recall that a space $X$ is said to be small generated if $\pi_1( X,x)=\pi_1^{sg}(X,x)$, also a space $X$ is said to be semilocally small generated if for every $x\in X$ there exists an open neighborhood $U$ of $x$ such that $i_*\pi_1(U,x) \leq \pi_1^{sg}(X,x)$. Torabi et al. \cite{Ts} proved that if $X$ is small generated space, then  $\pi_1^{qtop}(X,x)$ is an indiscrete
topological group and the quasitopological fundamental group of a semilocally small generated space is a topological group. By Theorem \ref{q1}, we obtain several results in quasitopological homotopy groups as follows:

\begin{corollary}
Let $X$ be a topological space such that $\Omega^{n-1}(X,x)$ is small generated. Then $\pi_n^{qtop}(X,x)$ is an indiscrete
topological group.
\end{corollary}

\begin{proof}
Since $\Omega^{n-1}(X,x)$ is a small generated space, then $\pi_1^{qtop}(\Omega^{n-1}(X,x), e_x)$ is an indiscrete
topological group, by \cite[Remark 2.11]{Ts}. Therefore $\pi_n^{qtop}(X,x)\cong \pi_1^{qtop}(\Omega^{n-1}(X,x), e_x)$ implies that $\pi_n^{qtop}(X,x)$ is an indiscrete
topological group.
\end{proof}

\begin{corollary}
Let $X$ be a topological space such that $\Omega^{n-1}(X,x)$ is a semilocally small generated space. Then $\pi_n^{qtop}(X,x)$ is a
topological group.
\end{corollary}

\begin{proof}
Since $\Omega^{n-1}(X,x)$ is semilocally small generated, then $\pi_1^{qtop}(\Omega^{n-1}(X,x), e_x)$ is a
topological group, by \cite[Theorem 4.1]{Ts}. Therefore $\pi_n^{qtop}(X,x)\cong \pi_1^{qtop}(\Omega^{n-1}(X,x), e_x)$ implies that $\pi_n^{qtop}(X,x)$ is a
topological group.
\end{proof}

Fabel \cite{F2}
proved that $\pi_1^{qtop}(HE,x)$ is not topological group. By considering the proof of this result it seems that if $\pi_1(X,x)$ is an abelian group, then $\pi_1^{qtop}(X,x)$ is a topological group.
He \cite{F3} also showed that for each $n \geq 2$ there exists a compact, path connected, metric
space X such that $\pi_n^{qtop}(X,x)$ is not a topological group. In the following example we show that there is a metric space $Y$ with abelian fundamental group such that $\pi_1^{qtop}(Y,y)$ is not a topological group.

\begin{example}
Let $n \geq 2$, $X$ be the compact, path connected, metric
space introduced in \cite{F3} such that $\pi_n^{qtop}(X,x)$ is not a topological group. By Theorem \ref{q1} $\pi_1^{qtop}(\Omega^{n-1}(X,x),  e_x)$ is not a topological group. Since for every $n \geq 2$, $\pi_n(X,x)$ is an abelian group, hence there is a metric space $Y = \Omega^{n-1}(X,x)$ with abelian fundamental group such that $\pi_1^{qtop}(Y,y)$ is not a topological group.
\end{example}

In \cite[Proposition 3.25]{Bra},  it is proved
that the quasitopological fundamental groups of shape injective spaces are Hausdorff. By Theorem \ref{q1} we have the following result.

\begin{corollary}\label{co}
Let $X$ be a topological space such that $\Omega^{n-1}(X,x)$ is shape injective space. Then $\pi_n^{qtop}(X,x)$ is Hausdorff.
\end{corollary}

\begin{proposition}\cite{Ts}
For a pointed topological space $(X,x)$, if $\{[e_x]\}$ is closed (or equivalently the topology of $\pi_1^{qtop}(X,x)$ is $T_0$), then $X$ is homotopically Hausdorff.
\end{proposition}

We generalized the above proposition as follows:

\begin{proposition}\label{13}
For a pointed topological space $(X,x)$, if $\{[e_x]\}$ is closed (or equivalently the topology of $\pi_n^{qtop}(X,x)$ is $T_0$), then $X$ is n-homotopically Hausdorff.
\end{proposition}

\begin{proof}
By Theorem \ref{q1} since $\pi_n^{qtop}(X,x)$ is $T_0$, hence $\pi_1^{qtop}(\Omega^{n-1}(X,x), e_x)$ is $T_0$. Therefore by previous proposition $\Omega^{n-1}(X,x)$ is homotopically Hausdorff which implies that $X$ is n-homotopically Hausdorff by \cite[Lemma 3.5]{Na}.
\end{proof}

\begin{corollary}
Let $X$ be a topological space such that $\Omega^{n-1}(X,x)$ is shape injective space. Then $X$ is n-homotopically Hausdorff.
\end{corollary}
\begin{proof}
It follows from Corollary \ref{co} and Proposition \ref{13}.
\end{proof}

Let $(B,b_0)$ be a pointed space and $p:E\longrightarrow B$ be a fibration with fiber $F$. Consider its mapping fiber, $Mp=\{(e,\omega)\in E\times B^I: \omega(0)=b_0 \ \ \text{and} \ \ \omega(1)=p(e) \}$. If  $\tilde{b_0}\in p^{-1}(b_0)$, then the injection map $k:\Omega(B,b_0)\longrightarrow Mp$ given by $k(\omega)=(\tilde{b_0}, \omega)$ induces a homomorphism $f_*: \pi_{n}(B,b_0)\longrightarrow  \pi_{n-1}(F,\tilde{b_0})$ \cite{Ro}.
\begin{theorem}
Let $(B,b_0)$ be a pointed space and $p:E\longrightarrow B$ be a fibration. If $\tilde{b_0}\in p^{-1}(b_0)$, then $f_*: \pi_{n}^{qtop}(B,b_0)\longrightarrow  \pi_{n-1}^{qtop}(F,\tilde{b_0})$ is continuous, for all $n\geq 1$.
\end{theorem}
\begin{proof}
We consider  the following commutative diagram:
\begin{equation}\label{dia}\begin{CD}
\Omega^{n-1}(\Omega(B,b_0), e_{b_0})@>k_{\sharp}>>\Omega^{n-1}(Mp,*)\\
@VV q V@V qVV\\
\pi_{n-1}^{qtop}(\Omega(B,b_0),  e_{b_0})@>k_*>>\pi_{n-1}^{qtop}(Mp,*),
\end{CD}\end{equation}
where $q$ is the quotient  map and $k_{\sharp}$ is the induced map of $k:\Omega(B, b_0)\longrightarrow Mp$ by the functor $\Omega^{n-1}$. Since $k_{\sharp}$ is continuous and $q$ is a quotient map, $k_*: \pi_{n-1}^{qtop}(\Omega(B,b_0),  e_{b_0})\longrightarrow \pi_{n-1}^{qtop}(F,\tilde{b_0})$ is continuous. By Theorem \ref{q1}, $\pi_{n-1}^{qtop}(\Omega(B,b_0),  e_{b_0})$ is isomorphic to $\pi_{n}^{qtop}(B,b_0)$. Therefore $f_*: \pi_{n}^{qtop}(B,b_0)\longrightarrow  \pi_{n-1}^{qtop}(F,\tilde{b_0})$ is continuous.
\end{proof}
\section{{\bf LONG EXACT SEQUENCE OF $\pi_{n}^{qtop}(X)$}}

Brazas \cite[Theorem 2.49]{Br3} proved that for every based pair (X, A) with inclusion $i:A\longrightarrow X$, there is a
long exact sequence in the category of quasitopological groups as follows:
\begin{align*}
...&\longrightarrow\pi_{n+1}^{qtop}(A)\rightarrow \pi_{n+1}^{qtop}(X)\rightarrow \pi_{n+1}^{qtop}(X,A)\rightarrow \pi_{n}^{qtop}(A)\longrightarrow ...\\
&\cdots\longrightarrow\pi_1^{qtop}(X)\longrightarrow\pi_1^{qtop}(X,A)\longrightarrow\pi_0^{qtop}(A)\longrightarrow\pi_0^{qtop}(X).
\end{align*}
He \cite[Proposition 2.20]{Br3} also showed that for every fibration $p:E\longrightarrow B$ of path connected spaces with
fiber $F$, there is a long exact sequence in the category of quasitopological groups as follows:
\begin{multline}\label{seq}
\cdots \longrightarrow\pi_n^{qtop}(E)\rightarrow \pi_n^{qtop}(B)\rightarrow \pi_{n-1}^{qtop}(F)\rightarrow \pi_{n-1}^{qtop}(E)\longrightarrow \cdots\\
\cdots\longrightarrow\pi_1^{qtop}(B)\longrightarrow\pi_0^{qtop}(F)\longrightarrow\pi_0^{qtop}(E)\longrightarrow\pi_0^{qtop}(B).
\end{multline}

 In follow, we obtain some results and examples by these exact sequences.
\begin{example}
Consider the pointed pair $(HA,HE)$, where $HA$ is the harmonic archipelago and $HE$ is the hawaiian earring. Then by \cite[Theorem 2.49]{Br3}, there is a long exact sequence in qTop:
\begin{align*}
...&\longrightarrow\pi_{n+1}^{qtop}(HE)\overset {}{\rightarrow} \pi_{n+1}^{qtop}(HA)\overset {}{\rightarrow} \pi_{n+1}^{qtop}(HA,HE)\overset {}{\rightarrow} \pi_{n}^{qtop}(HE)\longrightarrow ...\\
&\cdots\longrightarrow\pi_1^{qtop}(HA)\longrightarrow\pi_1^{qtop}(HA,HE)\longrightarrow\pi_0^{qtop}(HE)\longrightarrow\pi_0^{qtop}(HA).
\end{align*}

\end{example}

 Recall that a short exact sequence $E: 0\longrightarrow H\overset {i}{\rightarrow}  X \overset {\pi}{\rightarrow}  G\longrightarrow 0$ of topological abelian groups will be called an extension of topological groups if both $i$ and $\pi$ are continuous and open homomorphisms when considered as maps onto their images. Also, the extension $E$ is called split if and only if it is equivalent to the trivial extension $E_0: 0\longrightarrow H\overset {i_H}{\rightarrow}  H\times G \overset {\pi_G}{\rightarrow}  G\longrightarrow 0$ \cite{Be}.

\begin{theorem}\cite[Theorem 1.2]{Be}
Let $E: 0\longrightarrow H\overset {i}{\rightarrow}  X \overset {\pi}{\rightarrow}  G\longrightarrow 0$  be an extension of topological abelian groups. The following are equivalent:\\
(1)$E$ splits.\\
(2)There exists a right inverse for $\pi$.\\
(3)There exists a left inverse for $i$.
\end{theorem}

The above results hold for quasitopological groups, too.
\begin{proposition}
If $r:X\longrightarrow A$ is a retraction, then there are isomorphisms in quasitopological groups, for all $n\geq 2$,
\[\pi_n^{qtop}(X)\cong \pi_n^{qtop}(A)\times \pi_n^{qtop}(X,A).\]
\end{proposition}
\begin{proof}
Consider the pointed pair $(X,A)$. By \cite[Theorem 2.49]{Br3}, there is a long exact sequence
\begin{align*}
\cdots\longrightarrow\pi_{n+1}^{qtop}(X)\longrightarrow \pi_{n+1}^{qtop}(X,A)\longrightarrow\pi_{n}^{qtop}(A)\overset {i_*}{\rightarrow} \pi_{n}^{qtop}(X)\longrightarrow \pi_{n}^{qtop}(X,A)\longrightarrow \cdots.
\end{align*}
Since $r$ is a retraction  and $i_*$ is an injection, there is a short exact sequence
\[0\longrightarrow \pi_n^{qtop}(A)\overset {i_*}{\rightarrow}  \pi_n^{qtop}(X) \overset {}{\rightarrow}  \pi_n^{qtop}(X,A)\longrightarrow 0.\] Moreover, this sequence is an extension. Indeed, the map $i_*$ and $\pi_*$ are continuous and open homomorphisms when considered as maps onto their images. Therefore \[\pi_n^{qtop}(X)\cong \pi_n^{qtop}(A)\times \pi_n^{qtop}(X,A).\]
\end{proof}
\begin{proposition}
Let $B\subseteq A\subseteq X$ be pointed spaces. Then there is a long exact sequence of the triple $(X,A,B)$ in qTop:
\[\cdots\longrightarrow\pi_{n+1}^{qtop}(X,A)\longrightarrow\pi_{n}^{qtop}(A,B)\longrightarrow \pi_{n}^{qtop}(X,B)\longrightarrow\pi_{n}^{qtop}(X,A)\longrightarrow \pi_{n-1}^{qtop}(A,B)\longrightarrow\cdots .\]
\end{proposition}
\begin{proof}
Consider the following commutative diagram and chase a long diagram as follows:
\begin{center}
\includegraphics[height=5cm,width=11cm]{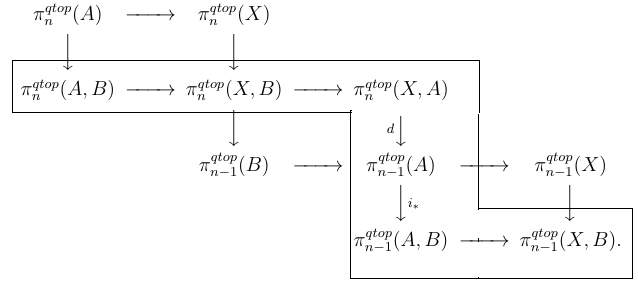}
\end{center}
\end{proof}

The following results are immediate consequences of Sequence (\ref{seq}).
\begin{corollary}\label{11}
If $p:E\longrightarrow B$ is a fibration with $E$ contractible, then \linebreak $f_*: \pi_{n}^{qtop}(B,b_0)\longrightarrow  \pi_{n-1}^{qtop}(F,\tilde{b_0})$ is an isomorphism in quasitopological groups for all $n\geq 2$ and $f_*: \pi_{1}^{qtop}(B,b_0)\longrightarrow  \pi_{0}^{qtop}(F)$ is an isomorphism in Set.
\end{corollary}
\begin{corollary}
Let $(X,x)$ be a pointed topological space. Then \linebreak $\pi_n^{qtop}(X,x)\cong \pi_{n-1}^{qtop}(\Omega(X,x),e_x)$ in quasitopological groups for all $n\geq 2$,  where  $e_x$ is the constant loop in $X$ at $x$ and $\pi_1^{qtop}(X,x)\cong \pi_{0}^{qtop}(\Omega(X,x))$  in Set.
\end{corollary}
\begin{proof}
By \cite[Proposition 4.3]{Sw}, the map $p:PX\longrightarrow X$ is a fibration with fiber $\Omega(X,x)$, where $PX= (X,x)^{(I,0)}$. By \cite[Proposition 2.20]{Br3}, the sequence
\begin{align*}
\cdots &\longrightarrow\pi_n^{qtop}(PX, e_x)\longrightarrow\pi_n^{qtop}(X,x)\longrightarrow\pi_{n-1}^{qtop}(\Omega (X,x), e_x)\longrightarrow\pi_{n-1}^{qtop}(PX, e_x)\longrightarrow \cdots\\
&\cdots\longrightarrow\pi_1^{qtop}(X,x)\longrightarrow\pi_0^{qtop}(\Omega (X,x))\longrightarrow\pi_0^{qtop}(PX)\longrightarrow\pi_0^{qtop}(X)
\end{align*}
is exact in qTop. By \cite[Proposition 4.4]{Sw}, $(PX, e_x)$ is contractible and therefore the result holds by Corollary \ref{11}.
\end{proof}
\begin{corollary}
If $p:\tilde{X}\longrightarrow X$ is a covering projection, then for all $n\geq 2$, $\pi_n^{qtop}(\tilde{X})\cong \pi_n^{qtop}(X)$ in quasitopological groups and $\pi_1^{qtop}(\tilde{X})$ can be embedded in $\pi_1^{qtop}(X)$.
\end{corollary}
\begin{proof}
This result follows by Sequence (\ref{seq}) and this fact that  the fiber $F$ of the covering projection $p$ is discrete and therefore $\pi_n^{qtop}(F)$ is trivial, for all $n\geq 1$.
\end{proof}

\end{document}